\documentclass[11pt]{amsart}
\usepackage{amsfonts,latexsym,amsthm,amssymb,graphicx}
\usepackage[all]{xy}
\usepackage{hyperref}
\usepackage[usenames]{color} 
\usepackage{manfnt}

\DeclareFontFamily{OT1}{rsfs}{}
\DeclareFontShape{OT1}{rsfs}{n}{it}{<-> rsfs10}{}
\DeclareMathAlphabet{\mathscr}{OT1}{rsfs}{n}{it}

\CompileMatrices

\newtheorem{theorem}{Theorem}[section]
\newtheorem{lemma}[theorem]{Lemma}

{\theoremstyle{definition} }
{\theoremstyle{remark} \newtheorem{remark}[theorem]{Remark}
}

\newcommand{\Pbb}{{\mathbb{P}}}
\newcommand{\Rbb}{{\mathbb{R}}}
\newcommand{\Zbb}{{\mathbb{Z}}}

\DeclareMathOperator{\lct}{lct}

\newcommand{\qede}{\hfill$\lrcorner$}

\title{
Log canonical threshold, Segre classes, and polygamma functions
}
\author{Paolo Aluffi}
\address{
Mathematics Department, 
Florida State University,
Tallahassee FL 32306, U.S.A.
}
\email{aluffi@math.fsu.edu}

\begin{document}

\begin{abstract}
We express the Segre class of a monomial scheme in projective space in terms of log 
canonical thresholds of 
associated ideals. Explicit instances of the relation amount to identities involving the
classical polygamma functions.
\end{abstract}

\maketitle


\section{Introduction}\label{intro}
As proven by J.~Howald (\cite{MR1828466}), the {\em log canonical threshold\/}
of a monomial ideal $I$ in a polynomial ring has a very simple expression 
in terms of the Newton diagram of the ideal: it measures the distance 
of the diagram from the origin along the main diagonal. It easily follows that 
the whole diagram for $I$ may be reconstructed from knowledge of the log canonical
thresholds of suitable extensions of the ideal, and hence other invariants of $I$
may be computed by using the thresholds of such extensions. We apply this simple 
observation to the {\em Segre class\/} of the scheme defined by $I$ in projective space.
The result is the following.

\begin{theorem}\label{lctsegre}
Let $I$ be a proper monomial ideal in the variables $x_1,\dots,x_n$, and let $S$ be 
the subscheme defined by $I$ in $\Pbb^N$, $N\ge n-1$. For $r_i>0$, denote by 
$I_{r_1,\dots, r_n}$ the extension of $I$ under the homomorphism defined by 
$x_i \mapsto x_i^{r_i}$, $i\le n$. Then
\begin{equation}\label{eq:main}
s(S,\Pbb^N)=1- \lim_{m\to \infty} \sum \frac{m\,n!\, X_1\cdots X_n}
{(m+a_1X_1+\cdots + a_n X_n)^{n+1}}
\end{equation}
where the sum is taken over all $(a_1,\dots, a_n)\in \Zbb_{>0}^n$ such that
\[
\lct(I_{a_2\cdots a_n,\dots, a_1\cdots a_{n-1}}) \ge \frac m{a_1\cdots a_n}\quad,
\]
and $X_i$ denotes the hyperplane $x_i=0$.\end{theorem}

The limit appearing in the statement should be interpreted as follows. When the
parameters $X_1,\dots, X_n$ are set to complex numbers (say, with positive real
part), the given limit converges to, and hence determines, a rational function 
of $X_1,\dots,X_n$, with a well-defined expansion as a series in $X_1,\dots, X_n$. 
The statement is that evaluating the terms of this series as intersection products
with $X_i=$ the $i$-th 
coordinate hyperplane in $\Pbb^N$, the right-hand side equals the Segre class of 
$S$ in $\Pbb^N$. (Each of the terms is supported on a subscheme of~$S$, 
cf.~Lemma~2.10 in \cite{proof}, hence this computation determines a class
in $A_*S$.) 

Theorem~\ref{lctsegre} is proved in~\S\ref{proof}. In~\S\ref{examples} we 
illustrate the result in simple examples. In the case of ideals generated by a pure 
power $x_1^\ell$, the statement reduces to an elementary limit of polygamma 
functions. In general, every independent computation of the Segre class of a
monomial ideal would give rise, via~\eqref{eq:main}, to an identity involving limits 
and series of such functions. We find this observation intriguing, but we hasten to 
add that the shape of the formulas, more than their algebro-geometric content, 
seems to be responsible for this phenomenon. The role played by the log canonical
threshold is limited to the demarcation of the Newton polytope of $I$ in the positive 
orthant in $\Rbb^n$ (Lemma~\ref{Nfromlct}).

Our main interest in Theorem~\ref{lctsegre} stems from the fact that both sides
of~\eqref{eq:main} are defined for arbitrary homogeneous ideals in a polynomial
ring. It is natural to ask to what extent formulas such as~\eqref{eq:main} may
hold for non-monomial schemes, perhaps after a push-forward to projective space.
\smallskip

{\em Acknowledgments.} The author's research is partially supported by a
Simons collaboration grant.


\section{Examples}\label{examples}
Let $n=1$, and $I=(x_1^\ell)$ for some $\ell\ge 1$. Then $I_{a_2\cdots a_n,
\dots, a_1\cdots a_{n-1}}=I$, $\lct (I)=\frac 1\ell$, and the range of summation 
specified in Theorem~\ref{lctsegre} is
$\lct(I)\ge \frac m{a_1}$, that is, $a_1\ge m\ell$.
Thus, the summation in the statement of the theorem is
\[
\sum_{a\ge m\ell} \frac{m X_1}{(m+a X_1)^2}\quad.
\]
Recall that the $r$-th {\em polygamma function $\Psi^{(r)}(x)$,\/} defined 
for $r>0$ as the $r$-th derivative of the {\em digamma\/} function
$
\frac d{dx} \ln(\Gamma(x)) = \frac{\frac d{dx} \Gamma(x)}{\Gamma(x)}
$,
admits the series representation
\[
\Psi^{(r)}(x)=(-1)^{r+1}r! \sum_{a\ge 0} \frac 1{(a+x)^{r+1}}
\]
for $x$ complex, not equal to a negative integer. We have
\[
\sum_{a\ge m\ell} \frac {x^2}{(m+ax)^2}= 
\sum_{a\ge 0} \frac {x^2}{(m+(a+m\ell)x)^2}= 
\sum_{a\ge 0} \frac {1}{(a+m\ell+\frac mx)^2}= 
\Psi^{(1)}\left(m\ell+\frac mx\right)\quad.
\]
Thus, formally
\[
\sum_{a\in \Zbb_{>0}, a\ge m\ell} \frac{m X_1}{(m+a X_1)^2}=
\frac {m \Psi^{(1)}(m\ell + \frac m{X_1})}{X_1}\quad,
\]
and the right-hand side in~\eqref{eq:main} may be rewritten as
\[
1-\lim_{m\to\infty}\frac m{X_1} \Psi^{(1)}\left(m\ell + \frac m{X_1}\right)\quad.
\]
The asymptotic behavior of $\Psi^{(r)}(x)$ is well-known: as $x\to \infty$
in any fixed sector not including the negative real axis,
\[
\Psi^{(r)}(x)\sim (-1)^{r+1} r! \left(\frac{x^{-r}}r+\frac{x^{-r-1}}2
+\sum_{k\ge 1} \frac{B_{2k}}{(2k)!} \frac{\Gamma(r+2k)}{\Gamma(r+1)} x^{-r-2k}
\right)
\]
(see for instance~\cite{DLMF}, (25.11.43)). In particular, for fixed $\ell$ and $x$
\[
\Psi^{(1)}\left(m\left(\ell+\frac 1x\right)\right)\sim 
\left(m\left(\ell+\frac 1x\right)\right)^{-1}=\frac{x}{m(1+\ell x)}
\]
as $m\to \infty$ in $\Zbb_{>0}$. Therefore,
\[
\lim_{m\to\infty}\frac m{X_1} \Psi^{(1)}\left(m\ell + \frac m{X_1}\right)
=\lim_{m\to\infty}\frac m{X_1}\frac{X_1}{m(1+\ell X_1)}
=\frac 1{1+\ell X_1}\quad.
\]
Theorem~\ref{lctsegre} asserts that
\[
s(S,\Pbb^N)=1-\frac 1{1+\ell X_1}= \frac{\ell X_1}{1+\ell X_1}=
c(N_S\Pbb^N)^{-1}\cap [S]\quad,
\]
as it should, since $S$ is a divisor in this case.

The assumption $n=1$ in this computation must be irrelevant, 
since the Segre class is not affected by this choice. The computation
itself {\em is,\/} however, affected by the choice of~$n$. Viewing the monomial
$x_1^\ell$ as a monomial in (for example) two variables $x_1,x_2$
leads via Theorem~\ref{lctsegre} to the formula
\[
s(S,\Pbb^N)=1-\lim_{m\to\infty} \sum \frac{2m X_1X_2}{(m+a_1 X_1+a_2X_2)^3}
\quad,
\]
where the summation is over all positive integers $a_1,a_2$ such that
$\lct(I_{a_2,a_1}) \ge \frac m{a_1a_2}$.
Since $I_{a_2,a_1}=(x_1^{\ell a_2})$, this amounts to the requirement that 
$a_1\ge m\ell, a_2\ge 1$, so the summation may be rewritten
\[
\sum_{a_1\ge m\ell, a_2\ge 1} \frac{2mX_1 X_2}{(m+a_1 X_1+a_2 X_2)^3}
=\frac {2mX_1X_2}{X_2^3} \sum_{a_1\ge m\ell} \sum_{a_2\ge 0} 
\frac 1{(a_2+1+\frac{m+a_1X_1}{X_2})^3}\quad.
\]
After performing the second summation, we see that the content of 
Theorem~\ref{lctsegre} in this case is
\begin{equation}\label{eq:inte}
s(S,\Pbb^N)=1-\lim_{m\to\infty} \frac {-mX_1}{X_2^2}\sum_{a_1\ge m\ell} 
\Psi^{(2)}\left(\frac{m+a_1X_1+X_2}{X_2}\right)\quad.
\end{equation}
Heuristically, we can now argue that, as $m\to \infty$,
\[
\Psi^{(2)}\left(\frac{m+a_1X_1+X_2}{X_2}\right)\sim -\left(\frac{m+a_1X_1+X_2}
{X_2}\right)^{-2}
\]
so that, again as $m\to \infty$,
\begin{multline*}
\sum_{a_1\ge m\ell} \Psi^{(2)}\left(\frac{m+a_1X_1+X_2}{X_2}\right)
\sim-\sum_{a_1\ge m\ell}\frac{X_2^2}{(m+a_1X_1+X_2)^2}
=-\frac{X_2^2}{X_1^2}\sum_{a\ge 0} \frac 1{(a+m\ell+\frac{m+X_2}{X_1})^2}\\
=-\frac{X_2^2}{X_1^2}\Psi^{(1)}\left(m\ell+\frac{m+X_2}{X_1}\right)
\sim -\frac{X_2^2}{X_1^2}\left(m\ell+\frac{m+X_2}{X_1}\right)^{-1}
=-\frac{X_2^2}{mX_1} \frac 1{1+\ell X_1+\frac {X_2}m}\quad.
\end{multline*}
Thus, the right-hand side of \eqref{eq:inte} equals
\[
1-\lim_{m\to \infty}  \frac 1{1+\ell X_1+\frac {X_2}m}=\frac{\ell X_1}{1+\ell X_1}
\]
as expected.

For `diagonal' ideals $I=(x_1^{\ell_1},\dots, x_n^{\ell_n})$, we have
\[
\lct(I_{a_2\cdots a_n,\dots,a_1\cdots a_{n-1}}) = \lct(x_1^{\ell_1 a_2\cdots a_n},\dots,
x_n^{a_1\cdots a_{n-1} \ell_n}) = \frac 1{\ell_1 a_2\cdots a_n} + \cdots 
+ \frac 1{a_1\cdots a_{n-1} \ell_n}\quad;
\]
the condition that this be $\ge m/a_1\cdots a_n$ is equivalent to
\[
\frac {a_1}{\ell_1} + \cdots + \frac {a_n}{\ell_n} \ge m\quad.
\]
For e.g., $n=2$, the content of Theorem~\ref{lctsegre} in this case is the identity
\begin{multline*}
1+\lim_{m\to \infty}\frac{m X_1}{X_2^2}\left(
\sum_{a_1=1}^{m \ell_1-1}\Psi^{(2)}\left(m\ell_2-\lfloor\frac{a_1\ell_2}{\ell_1}\rfloor
+\frac{m+a_1 X_1}{X_2}\right)
+\sum_{a_1\ge m\ell_1} \Psi^{(2)}\left(1+\frac{m+a_1 X_1}{X_2}\right)\right) \\
=\frac{\ell_1 \ell_2 X_1 X_2}{(1+\ell_1 X_1)(1+\ell_2 X_2)} \quad.
\end{multline*}


\section{Proof of Theorem~\ref{lctsegre}}\label{proof}
For positive integers $r_1,\dots,r_n$ and a homogeneous ideal $I$ of $k[x_1,\dots,x_{N+1}]$
generated by polynomials in $x_1,\dots, x_n$, with $N+1\ge n$, we let $I_{r_1,\dots,r_n}$ 
denote the extension of $I$ via the ring homomorphism $k[x_1,\dots,x_{N+1}] 
\to k[x_1,\dots,x_{N+1}]$ defined by
$x_i \mapsto x_i^{r_i}, i=1\dots, n$.
If $I$ is a monomial ideal, let $N'\subset \Rbb^n$ be the convex hull of the lattice 
points $(i_1,\dots,i_n)\in \Zbb^n$ such that $x_1^{i_1}\cdots x_n^{i_n}\in I$,
and let $N$ be the (closure of the) complement of $N'$ in the positive orthant
$\Rbb_{\ge 0}^n$. We call $N$ the `Newton region' for $I$.

If $I$ is monomial, the ideal $I_{r_1,\dots,r_n}$ is also monomial, and its Newton 
region is obtained by stretching $N$ by a factor of $r_1$ in the $x_1$ direction, 
\dots, $r_n$ in the $x_n$ direction. We will denote by $N_{r_1,\dots,r_n}$
this stretched region.

\begin{lemma}\label{Nfromlct}
Let $I$ be a proper monomial ideal, and let $N$ be as above.
For $(a_1,\dots, a_n)\in \Zbb_{>1}^n$ and $m>0$,
\[
\left(\frac {a_1}m,\dots, \frac {a_n}m\right)\in N \iff
a_1\cdots a_n \lct(I_{a_2\cdots a_n,\dots, a_1\cdots a_{n-1}})\le m\quad.
\]
\end{lemma}

\begin{proof}
Let $a_1,\dots, a_n$ integers $>1$. Note that
\begin{align*}
\left(\frac {a_1}m,\dots, \frac {a_n}m\right)\in N &\iff
\left(\frac {a_1}m a_2\cdots a_n,\dots, \frac {a_n}m a_1\cdots a_{n-1}\right)\in 
N_{a_2\cdots a_n,\dots, a_1\cdots a_{n-1}} \\
&\iff \frac{a_1\cdots a_n}m (1,\dots,1)\in 
N_{a_2\cdots a_n,\dots, a_1\cdots a_{n-1}}\quad.
\end{align*}
By Howald's result (\cite{MR1828466}, Example~5) this is the case if and only if
\[
\frac{a_1\cdots a_n}m \le \frac 1{\lct(I_{a_2\cdots a_n,\dots, a_1\cdots a_{n-1}})}\quad, 
\]
yielding the statement.
\end{proof}

\begin{remark}
The restriction to $a_i>1$ in this statement is in order to ward off the `annoying exception'
raised in~\cite{MR1828466}, Example~5: 
the formula for the log canonical threshold used in 
the proof does not hold if the corresponding multiplier ideal is trivial. 
In any case, the difference between $N$ and the region spanned by the $n$-tuples 
$(\frac {a_1}m,\dots, \frac {a_n}m)$ satisfying the stated condition with $a_i>0$
vanishes in the limit as $m\to \infty$, so we may (and will) adopt the condition for 
$(a_1,\dots, a_n) \in \Zbb^n_{>0}$ in the application to Theorem~\ref{lctsegre}.
\qede\end{remark}

By Lemma~\ref{Nfromlct}, the limit in~\eqref{eq:main} equals
\[
\lim_{m\to\infty}\frac 1{m^n}
\sum_{\substack{(a_1,\dots,a_n)\in \Zbb_{>0}^n \\
 (\frac{a_1}m,\dots,\frac{a_n}m)\in N'}}
\frac{n!\, X_1\cdots X_n}{(1+\frac{a_1}m X_1+\cdots+ \frac{a_n}m X_n)^{n+1}}\quad.
\]
This may be interpreted as a limit of Riemann sums for the integral
\[
\int_{N'} \frac{n! X_1\cdots X_n\, da_1\cdots da_n}
{(1+a_1 X_1+\cdots +a_n X_n)^{n+1}}\quad.
\]
Since the value of this integral on the positive orthant is~$1$, the right-hand side 
of~\eqref{eq:main} equals
\[
\int_N \frac{n! X_1\cdots X_n\, da_1\cdots da_n}
{(1+a_1 X_1+\cdots +a_n X_n)^{n+1}}\quad.
\] 
This equals the Segre class $s(S,\Pbb^N)$ once $X_i$ is interpreted as the $i$-th
coordinate hyperplane in $\Pbb^N$, by~Theorem~1.1 in~\cite{proof}.
\qede



\end{document}